\documentclass[12pt,reqno]{amsart}
\usepackage{amsmath,amsthm,amssymb,mathrsfs,enumerate,bm,enumitem}
\usepackage{comment}
\usepackage{color}
\usepackage{graphicx, color}

\title[On Highly-regular graphs]
{On Highly-regular graphs}

\author[T.~Kousaka]{Taichi Kousaka}
\address{Graduate School of Mathematics, Kyushu University,
744 Motooka, Nishi-ku, Fukuoka, 819-0395 JAPAN}
\email{t-kosaka@math.kyushu-u.ac.jp}


\newtheorem{theorem}{Theorem}[section]
\newtheorem{corollary}[theorem]{Corollary}

\newtheorem{proposition}[theorem]{Proposition}

\theoremstyle{definition}

\theoremstyle{remark}
\newtheorem{remark}[theorem]{Remark}

\newtheorem{question}{Question}

\DeclareMathOperator{\F}{\mathbb{F}}

\DeclareMathOperator{\Z}{\mathbb{Z}}

\DeclareMathOperator{\R}{\mathbb{R}}

\DeclareMathOperator{\I}{\mathrm{i}}
\DeclareMathOperator{\diam}{\mathrm{diam}}
\DeclareMathOperator{\Sp}{\mathrm{sp}}
\DeclareMathOperator{\m}{\mathrm{m}}

\begin{document}

\begin{abstract}
Highly-regular graphs can be regarded as a combinatorial generalization of distance-regular graphs. 
From this standpoint, we study combinatorial aspects of highly-regular graphs. 
As a result, we give the following three main results in this paper. 
Firstly, we give a characterization of a distance-regular graph by using the index and diameter of a highly-regular graph. 
Secondly, we give two constructions of highly-regular graphs. 
Finally, we generalize well-known properties of the intersection numbers of a distance-regular graph. 
\end{abstract}

\maketitle

\section{Introduction}
All graphs which we consider in this paper are finite undirected graphs without loops and multiple edges. For a graph $\Gamma$, we denote the vertex set of $\Gamma$ by $V(\Gamma)$, the edge set of $\Gamma$ by $E(\Gamma)$. 
For a connected graph $\Gamma$ and two vertices $u, v \in V(\Gamma)$, let $d(u, v)$ be the distance of $u$ and $v$, that is, the length of the shortest path from $u$ to $v$, and  $\diam(\Gamma)$ be the diameter of a graph $\Gamma$, that is, the maximum length of $d(u, v)$, $u, v \in V(\Gamma)$. 
If $\Gamma$ is not connected, the diameter of $\Gamma$ is defined as infinity. 
For a vertex $u$ and an integer $i \in \{0, \dots, \diam(\Gamma)\}$, let $D_{i}(u)$ or $D_{i, \Gamma}(u)$ be a set of all vertices which are at distance $i$ from $u$. 
Let $V(\Gamma)=\{v_{1}, \dots,v_{n}\}$ a labeling of $V(\Gamma)$ and $A$ be the adjacency matrix of a graph $\Gamma$ with spectrum $\Sp(\Gamma)=\{ \lambda_{0}^{\m(\lambda_{0})}, \dots, \lambda_{d}^{\m(\lambda_{d})}\}$, where $\lambda_{0}>\cdots >\lambda_{d}$, and the superscripts $m(\lambda_{i})$ stand for the multiplicities. Let $E_{i}$, $i=0, \dots, d$ be the minimal idempotents representing the orthogonal projections on the eigenspaces associated with $\lambda_{i}$. 
The above notation is used throughout this paper. 

Classically, many special regular graphs have been widely studied. An important class of regular graphs is the class of strongly-regular graphs. Here, a connected graph $\Gamma$ is {\it strongly-regular with parameters $(k, \alpha, \beta)$} if it is $k$-regular, for any $u, v \in V(\Gamma)$ with $\{u, v\} \in E(\Gamma)$, $|D_{1}(u)\cap D_{1}(v)|=\alpha$ and for any $u, v \in V(\Gamma)$ with $\{u, v\} \notin E(\Gamma)$, $|D_{1}(u)\cap D_{1}(v)|=\beta$. 
The class of strongly-regular graphs is much smaller than the class of regular graphs. 
There is a wider class which is contained in the class of regular graphs and contains the 
class of strongly-regular graphs. 
It is the class of {\it highly-regular graphs}. 
Here, a graph $\Gamma$ of order $n$ is {\it highly-regular with collapsed adjacency matrix $($CAM, for short$)$ $C=[c_{i,j}]_{1 \leq i, j \leq m}$ $(2 \leq m<n)$} if for every vertex $u \in V(\Gamma)$ there is a partition of $V(\Gamma)$ into $V_{1}(u)=\{u\}, V_{2}(u), \dots, V_{m}(u)$ such that each vertex $y \in V_{j}(u)$ is adjacent to $c_{i,j}$ vertices in $V_{i}(u)$. In this paper, we allow $m=n$ only if $n=2$. Namely, we regard the $1$-regular graph as a highly-regular graph. 
The class of highly-regular graphs was introduced as an interesting class of regular graphs by B.~Bollob{\'a}s ({\it cf.}~\cite{bollobas1}). 
Highly-regular graphs are a combinatorial generalization of strongly-regular graphs ({\it cf.}~\cite{bollobas}). 
From this standpoint, strongly-regular graphs are characterized by index of highly-regular graphs ({\it cf.}~\cite{alavi}). Here, the {\it index} of a highly-regular graph is the least positive integer of the size of CAMs. The index is an important invariant of highly-regular graphs. 

In addition to the above two classes, there are other classically well-known classes of regular graphs. 
One of these classes is the class of distance-transitive graphs. Here, a connected graph $\Gamma$ is {\it distance-transitive} if for every four vertices $u, v, x, y \in V(\Gamma)$ with $d(u, v)=d(x, y)$, there exists an automorphism $\sigma$ of $\Gamma$ such that $\sigma(x)=u$ and $\sigma(y)=v$ ({\it cf.}~\cite{brouwer}). 
One of the important and remarkable properties of distance-transitive graphs is that there are only finitely many distance-transitive graphs of fixed valency greater than $2$ ({\it cf.}~\cite{cameron}, \cite{cameronpraegersaxlseitz}). 
There are many important properties of distance-transitive graphs which are related to the other mathematical theory such as algebraic combinatorics. 
N.~Biggs introduced distance-regular graphs as a combinatorial generalization of distance-transitive graphs by observing that several combinatorial properties of distance-transitive graphs also hold in the class of distance-regular graphs. Here, a connected graph $\Gamma$ is {\it distance-regular} if the integers $|D_{1}(v) \cap D_{i-1}(u)|$, $|D_{1}(v) \cap D_{i}(u)|$, and $|D_{1}(v) \cap D_{i+1}(u)|$  depend only on $i=d(u, v) \in \{1, \dots, \diam(\Gamma) \}$ ({\it cf.}~\cite{brouwer}). 
In $1984$, E.~Bannai and T.~Ito conjectured that there are only finitely many distance-regular graphs of fixed valency greater than $2$ ({\it cf.}~\cite{bannai}). 
In $2015$, S.~Bang, A.~Dubickas, J.~H.~Koolen and V.~Moulton proved this conjecture conclusively ({\it cf.}~\cite{bang}). 

So far, we introduced four classes of regular graphs, that is, strongly-regular graphs, highly-regular graphs, distance-transitive graphs and distance-regular graphs. 
As we mentioned above, highly-regular graphs and distance-regular graphs were introduced in different contexts. 
However, there is a naturally connection. 
Distance-regular graphs are highly-regular graphs, that is, highly-regular graphs can be regarded as a combinatorial generalization of distance-regular graphs. From this standpoint, we study combinatorial aspects of highly-regular graphs. As a result, we give the following three main results in this paper.

Firstly, distance-regular graphs are characterized by index and diameter of highly-regular graphs. 
\begin{theorem}
A connected highly-regular graph $\Gamma$ has the index $\diam(\Gamma)+1$ if and only if it is a distance-regular graph.
\end{theorem}

Secondly, we give a construction of connected highly-regular graphs with diameter $2$ which are not distance-regular graphs. Moreover, we also give a construction of highly-regular graphs which do not always have diameter $2$ by using a symmetric association scheme. 
\begin{theorem}
For a highly-regular graph $\Gamma$ with $3 \leq \diam(\Gamma) < \infty$, the complement of the graph is a highly-regular graph with $\diam(\overline{\Gamma})=2$ which is not a distance-regular graph. 
Moreover, Let $\mathfrak{X}=(X, \{R_{i}\}_{i=0}^{d})$ be a symmetric association scheme of class $d$. Then, for each $l \in \{1, \dots, d\}$, the graph $\Gamma_{R_{l}}=(X, E_{R_{l}})$ is a highly-regular graph, where $E_{R_{l}}=\{\{x,y\} \mid (x,y) \in R_{l} \}$. 
\end{theorem}

Poulos showed that finite upper half plane graphs are highly-regular graphs ({\it cf.}~\cite{angel}, \cite{terras}). 
The construction of highly-regular graphs by using a symmetric association scheme in Theorem $1.2$ is a generalization of Poulos's result. 

Finally, we give a generalization of well-known properties of the intersection numbers of distance-regular graphs. 
\begin{theorem}
Let $\Gamma$ be a connected highly-regular graph with CAM $C=[c_{i,j}]_{1 \leq i, j \leq m}$ and valency $k$. Here, let a labeling of $C$ be a labeling with respect to distance. 
For each $i \in \{0, 1, \dots, \diam(\Gamma)\}$, there exists a nonempty subset $S_{i}$ of $I=\{1, \dots, m\}$ such that $D_{i}(u)=\bigsqcup_{t \in S_{i}}V_{t}(u)$, $u \in V$. 
For each $i \in \{1, \dots, \diam(\Gamma) \}$, let the integers $b_{i-1}^{\max}$,  $c_{i}^{\max}$, and $c_{i}^{\min}$ be the following: 
\begin{itemize}
\item $b_{i-1}^{\max}=\max \{ \sum_{t \in S_{i}}c_{t, l} \mid l \in S_{i-1} \}$.
\item $c_{i}^{\max}=\max \{ \sum_{t \in S_{i-1}}c_{t, l} \mid l \in S_{i} \}$. 
\item $c_{i}^{\min}=\min \{ \sum_{t \in S_{i-1}}c_{t, l} \mid l \in S_{i} \}$. 
\end{itemize} 
Then, the following inequalities hold: 
	\begin{enumerate}[label={\rm (\arabic*)}]
	\item $k=b_{0}^{\max} \geq b_{1}^{\max} \geq \cdots \geq b_{\diam(\Gamma)-1}^{\max} \geq 1$. 
	\item $1=c_{1}^{\min} \leq c_{2}^{\min} \leq \cdots \leq c_{\diam(\Gamma)}^{\min} \leq k$. 
	\end{enumerate}
Moreover, we suppose that $\Gamma$ satisfies the following property: 
$(\star)$ For any $u \in V(\Gamma)$, $i \in \{0, 1, \dots, \diam(\Gamma)-1 \}$, $x \in D_{i}(u)$, the set $D_{1}(x) \cap D_{i+1}(u)$ is nonempty set. 
Then, the following inequality holds: 
	\begin{enumerate}[label={\rm (\arabic*)}]
	\item[$(3)$] If $i \geq 1$ and $i+j \leq \diam(\Gamma)$, then $c_{i}^{\max} \leq b_{j}^{\max}$. 
	\end{enumerate}
\end{theorem}

It is well-known that distance-regular graphs have interesting combinatorial properties. 
Moreover, distance-regular graphs have many applications such as coding theory. 
On the other hand, the class of highly-regular graphs is very wide. 
For example, vertex-transitive graphs with non-identity stabilizers are highly-regular graphs. 
For such reasons, it seems that it is difficult to investigate interesting properties of highly-regular graphs. 
As we described above, however, highly-regular graphs have similar properties of distance-regular graphs. 
Therefore, we believe that  it is an interesting approach to study highly-regular graphs that highly-regular graphs are considered as a generalization of distance-regular graphs. 

\section{Preliminaries}
Let $\Gamma$ be a graph of order $n$. We denote the complement of $\Gamma$ by $\overline{\Gamma}$. 
If $\Gamma$ is a distance-regular graph, the complement of $\Gamma$ is not always a distance-regular graph. However, if $\Gamma$ is a strongly-regular graph and $\overline{\Gamma }$ is connected, $\overline{\Gamma}$ is also a strongly-regular graph.  
This property is generalized to a highly-regular graph as follows. 
 
\begin{proposition}$($cf.~{\rm [1, PROPOSITION 1]}$).$\label{1, prop1}
A graph $\Gamma$ is a highly-regular graph if and only if the complement of $\Gamma$ is a highly-regular graph. 
\end{proposition}

For two graphs $\Gamma_{1}$ and $\Gamma_{2}$, the Cartesian product $\Gamma_{1} \Box \Gamma_{2}$ of $\Gamma_{1}$ and $\Gamma_{2}$ is a graph such that 
the vertex set of $\Gamma_{1} \Box \Gamma_{2}$ is $V(\Gamma_{1}) \times V(\Gamma_{2})$, and 
two vertices $(u_{1}, v_{1})$, $(u_{2}, v_{2})$ are adjacent if and only if 
$u_{1}=u_{2}$ and $\{v_{1}, v_{2}\} \in E(\Gamma_{2})$, or
$v_{1}=v_{2}$ and $\{u_{1}, u_{2}\} \in E(\Gamma_{1})$. 
In general, the Cartesian product of two distance-regular graphs is not always a distance-regular graph. However, the Cartesian product of two highly-regular graphs is a highly-regular graph. 

\begin{proposition}$($cf.~{\rm [1, PROPOSITION 6]}$).$\label{1, prop6}
If $\Gamma_{1}$ and $\Gamma_{2}$ are highly-regular graphs, the Cartesian product $\Gamma_{1} \Box \Gamma_{2}$ is a highly-regular graph. 
\end{proposition}

Properties of a partition of a vertex set corresponding to a CAM are important in the class of highly-regular graphs. 
It is straight-forward to see that distance-regular graphs are highly-regular graphs. 
Therefore, highly-regular graphs are a combinatorial generalization of distance-regular graphs. The following expected properties are satisfied. 

\begin{proposition}$($cf.~{\rm [1, PROPOSITION 3]}$)$.\label{1, prop3}
Let $\Gamma$ be a connected highly-regular graph with CAM $C=[c_{ij}]_{1 \leq i,j \leq m}$. 
For $u, v \in V({\Gamma})$, let $V_{1}(u)=\{u\}$, $V_{2}(u)$, $\dots$, $V_{m}(u)$ and $V_{1}(v)=\{v\}$, $V_{2}(v)$,  $\dots$, $V_{m}(v)$ be the corresponding partitions of $V(\Gamma)$. Then the following properties are satisfied. 
\begin{enumerate}[label={\rm (\arabic*)}]
\item For each $i \in \{0, 1, \dots, \diam(\Gamma) \}$, there exists a nonempty subset $S_{i}$ of $I=\{1,2, \dots, m\}$ such that $D_{i}(u)=\bigsqcup_{t \in S_{i}}V_{t}(u)$, $D_{i}(v)=\bigsqcup_{t \in S_{i}}V_{t}(v)$. 
\item For $t \in \{1, 2, \dots, m\}$, $|V_{t}(u)|=|V_{t}(v)|$. 
\item For $i \in \{0, 1, \dots, \diam(\Gamma) \}$, the induced subgraph of $D_{i}(u)$ and the induced subgraph of $D_{i}(v)$ have the same degree sequence. 
\end{enumerate}
\end{proposition}

In Section $1$, we introduced four classes of regular graphs, that is, strongly-regular graphs, highly-regular graphs, distance-transitive graphs and distance-regular graphs. Now, we introduce an invariant of highly-regular graphs ({\it cf}.~\cite{alavi}). 

Let $\Gamma$ be a highly-regular graph of order $n$. {\it The index of }$\Gamma$ is the least positive integer $m$ such that a CAM has the size $m$ $(2 \leq m < n)$, and we denote by $\I(\Gamma)$. 
By Proposition \ref{1, prop1}, the index of a highly-regular graph is equal to that of the complement of the graph. By Proposition \ref{1, prop3}, the index $\I(\Gamma)$ is greater than $\diam(\Gamma)$ if the graph $\Gamma$ is connected. Moreover, we get a lower bound of $\I(\Gamma)$. 

Let $\Gamma$ be a connected highly-regular graph. 
For any vertex $u \in V({\Gamma})$ and for any $i \in \{1, \dots, \diam(\Gamma) \}$, 
we denote the induced subgraph of $D_{i}(u)$ by $\langle D_{i}(u) \rangle$. Then, we have the following proposition. 
\begin{proposition}$($cf. {\rm [1, PROPOSITION 5]}$).$\label{1, prop5}
If the cardinality of the degree set of $\langle D_{i}(u) \rangle$ $(1 \leq i \leq \diam(\Gamma))$ is $k_{i}$, then
	\begin{align*}
		\I(\Gamma) \geq 1+\sum_{i=1}^{\diam(\Gamma)}k_{i}.
	\end{align*}
\end{proposition}

By the above discussion, we have the following inequalities. 
\begin{align*}
	\I(\Gamma) \geq 1+\sum_{i=1}^{\diam(\Gamma)}k_{i} \geq 1+ \diam(\Gamma). 
\end{align*}
Here, we note that we can easily find highly-regular graphs which do not attain the above equalities. 

\section{A characterization of distance-regular graphs by using index and diameter of highly-regular graphs}

In Section $2$, we introduced index of highly-regular graphs. 
By using the index, we can characterize a strongly-regular graph, that is, a connected highly-regular graph has the index $3$ if and only if it is a strongly-regular graph ({\it cf}.~\cite{alavi}). 
In this section, we characterize a distance-regular graph by using the index. 

\begin{theorem}\label{characterization thm1}
A graph $\Gamma$ is a connected highly-regular graph with the valency $k$ and CAM of the form (up to a labeling) 
	\begin{align*}
	\left(
	\begin{array}{cccccc}
		0	&	1	&	  0     	& \cdots	&	0	&	0 \\
		k	&	a_{1} &	c_{2} 	&	~	& \vdots &	\vdots \\
		0	&	b_{1} &	\ddots 	& \ddots &    0	&   \vdots \\
	   \vdots & 	0	& 	\ddots 	& \ddots &   c_{m-2} &	0 \\
	   \vdots & \vdots & 	~		&  ~		& a_{m-2} & c_{m-1} \\
		0	&	0	&	\cdots 	&	0	& b_{m-2} & a_{m-1} 
	\end{array}
	\right)
	\end{align*}
if and only if it is a distance-regular graph with $\diam(\Gamma)=m-1$. 
\end{theorem}
\begin{proof}

Let $\Gamma$ be a connected highly-regular graph which satisfies the above condition. 
For any $v \in V({\Gamma})$, there exists a partition of $V({\Gamma})$ with respect to the above CAM. Let $V_{0}(v)=\{v\}, V_{1}(v), \dots, V_{m-1}(v)$ be the partition of $V(\Gamma)$. 
First, we have $D_{0}(v)=\{v\}=V_{0}(v)$. 
By the first column of the CAM, we have $D_{1}(v)=V_{1}(v)$. 
Then, by the second column of the CAM, we have $D_{2}(v)=V_{2}(v)$. 
By repeating this argument, we have $D_{\diam(\Gamma)}=V_{m-1}(v)$. 
Hence, $m=\diam(\Gamma) +1$ and $\Gamma$ is a distance-regular graph. 

Conversely, we suppose a graph $\Gamma$ is a distance-regular graph with $\diam(\Gamma)=m-1$. 
By the definition of a distance-regular graph and the index, we have $1+\diam(\Gamma) \geq \I(\Gamma)$. 
By $\I(\Gamma) \geq 1+\diam(\Gamma)$, we have $\I(\Gamma)=m$. 
Therefore, we have the desired result. 		

\end{proof}

By Theorem \ref{characterization thm1}, we conclude the following characterization of a distance-regular graph by using the index. 

\begin{corollary}\label{characterization cor1}
Let $\Gamma$ be a connected highly-regular graph. Then, the graph $\Gamma$ has the index $\diam(\Gamma)+1$ if and only if it is a distance-regular graph. 
\end{corollary}

This is very useful to determine whether a highly-regular graph is a distance-regular graph. 

\section{A construction of highly-regular graphs with diameter $2$ which are not distance-regular graphs}

In Section $3$, we characterize a distance-regular graph by the index of a highly-regular graph. 
In this section, we construct highly-regular graphs with diameter $2$ which are not distance-regular graphs. 

Let $\Gamma$ be a highly-regular graph with $2 \leq \diam(\Gamma) < \infty$. 

First, we consider the case where $\Gamma$ is a distance-regular graph. By Corollary \ref{characterization cor1}, the graph $\Gamma$ has the index $\diam(\Gamma)+1$. 

If $\diam(\Gamma)$ is equal to $2$, the graph $\Gamma$ and the complement $\overline{\Gamma}$ have the index $3$. In this case, if $\overline{\Gamma}$ is connected, both $\Gamma$ and $\overline{\Gamma}$ are strongly-regular graphs. 

If $\diam(\Gamma)$ is greater than $2$, $\diam(\overline{\Gamma})$ is equal to $2$. 
Here, we note that for a connected regular graph with the diameter greater than $2$, the diameter of the complement is $2$. Hence, we have the following inequality. 
\begin{align*}
	\I(\Gamma) =1+\diam(\Gamma) > 1+\diam(\overline{\Gamma}).
\end{align*}
By Proposition \ref{1, prop1}, the index of $\Gamma$ is equal to the index of $\overline{\Gamma}$. Therefore, we have the following inequality. 
\begin{align*}
\I(\overline{\Gamma}) > 1+\diam(\overline{\Gamma}). 
\end{align*}
By Corollary \ref{characterization cor1}, the graph $\overline{\Gamma}$ is a highly-regular graph which is not a distance-regular graph with $\diam(\overline{\Gamma})=2$. 

Next, we consider the case where the graph $\Gamma$ is not a distance-regular graph. 
If $\diam(\Gamma)$ is equal to $2$, the index of $\overline{\Gamma}$ is greater than $3$. 
If $\diam(\Gamma)$ is greater than $2$, the index of $\overline{\Gamma}$ is greater than $3$. Therefore, the graph $\overline{\Gamma}$ is a highly-regular graph which is not a distance-regular graph with $\diam(\overline{\Gamma})=2$. 

By the above discussion, we conclude the following theorem. 
\begin{theorem}\label{construction thm1}
For a highly-regular graph $\Gamma$ with $3 \leq \diam(\Gamma) < \infty$, the complement of the graph is a highly-regular graph with $\diam(\overline{\Gamma})=2$ which is not a distance-regular graph. 
\end{theorem}

By taking the complement of a distance-regular graph with the diameter greater than $2$, we obtain a highly-regular graph with the diameter $2$ which is not a distance-regular graph. 

\section{Another construction of highly regular graphs by using a symmetric association schemes}

In Section $4$, we gave a construction of highly-regular graphs which are not distance-regular graphs. However, this construction can generate only highly-regular graphs with diameter $2$ which are not distance-regular graphs. In this section, we give another construction of highly-regular graphs by using a symmetric association scheme. 

Let $X$ be a finite set and $R_{i}$ $(i=0, \dots, d)$ be nonempty subsets of $X \times X$. 
A {\it symmetric association scheme of class $d$} is a pair $\mathfrak{X}=(X, \{R_{i}\}_{i=0}^{d})$ satisfying the following conditions: 
\begin{enumerate}[label={\rm (\arabic*)}]
\item[(SAS-1)] $R_{0}=\{ (x, x) \mid x \in X\}$. 
\item[(SAS-2)] $X \times X = \bigsqcup_{i=0}^{d}R_{i}$. 
\item[(SAS-3)] $^{t}R_{i}=R_{i}$ for any $i \in \{0, \dots, d\}$, where $^{t}R_{i}=\{ (y, x) \mid (x, y) \in R_{i} \}$. 
\item[(SAS-4)] for any $i, j, l \in \{1, \dots, d\}$, there exists constants $p_{i,j}^{l}$ such that for all $x, y \in X$ with $(x, y) \in R_{l}$, $p_{i, j}^{l}=|\{z \in X \mid (x, z) \in R_{i}$ and $(z, y) \in R_{j}\}|$. 
\end{enumerate}
The above constants $p_{i,j}^l$ are called the {\it intersection numbers}. For each $i \in \{0, \dots, d\}$, we denote the matrix $[p_{i,j}^{l}]_{0 \leq j,l \leq d}$ by $B_{i}$. The matrix $B_{i}$ is called the {\it $i$-th intersection matrix} of $\mathfrak{X}$. For any $x \in X$ and $R_{i}$, we denote a set of  elements $y \in X$ with $(x, y) \in R_{i}$ by $xR_{i}$. 

Let $\mathfrak{X}=(X, \{R_{i}\}_{i=0}^{d})$ be a symmetric association scheme of class $d$.
Then, we have the following theorem. 

\begin{theorem}\label{construction thm2}
For each $l \in \{0, \dots, d\}$, the graph $\Gamma_{R_{l}}=(X, E_{R_{l}})$ is a highly-regular graph with CAM $B_{l}$, where $E_{R_{l}}=\{\{x,y\} \mid (x,y) \in R_{l} \}$.  
\end{theorem}
\begin{proof}
We take arbitrary $x \in X$. By (SAS-2), we get the decomposition of $X$ as follows: 
	\begin{align*}
		X=\bigsqcup_{i=0}^{d}xR_{i}. 
	\end{align*}
We take any two partitions $xR_{i}, xR_{j}$ of the above decomposition. By (SAS-4), for any $y \in xR_{j}$, the number of vertices $z \in xR_{i}$ such that $\{y, z\} \in E_{R_{l}}$ is equal to $p_{i,l}^{j}$. The integer $p_{i,l}^{j}$ is independent of choice of an element $(x, y) \in R_{j}$. In particular, the integer $p_{i,l}^{j}$ is independent of choice of $y \in xR_{j}$. Therefore, the graph $\Gamma_{R_{l}}$ is a highly-regular graph with CAM $B_{l}$. 
\end{proof}

Moreover, let $G$ be a finite group and $G$ acts on $X$ transitively. Let $S=\{\Delta_{0}, \dots, \Delta_{d} \}$ be the set of $G$-orbits of $X \times X$. We suppose that each $G$-orbit of $X \times X$ is symmetric. Then, $\mathfrak{X}=(X, S)$ is a symmetric association scheme of class $d$. This symmetric association scheme is closely related to harmonic analysis of finite homogeneous spaces ({\it cf.}~\cite{bannai}, \cite{cecc}). By Theorem \ref{construction thm2}, we have the following corollary. 

\begin{corollary}\label{construction thm3}
For each $l \in \{0, \dots, d\}$, the graph $X_{\Delta_{l}}$ is a highly-regular graph with CAM $B_{l}$. 
\end{corollary}

By Theorem \ref{construction thm2} and Corollary \ref{construction thm3}, we get many examples of highly-regular graphs which are not always distance-regular graphs. 
For example, Euclidean graphs and finite upper half plane graphs are highly-regular graphs ({\it cf}.~\cite{terras}). 
Generalized Euclidean graphs are highly-regular graphs ({\it cf}.~\cite{bannai2009}, \cite{kwok}). Moreover, the other graphs which appear in \cite{kwok} are also. 

In the rest of this section, we construct graphs which are a special case of graphs defined by W.~Li in \cite{winnie}. 
Let $p$ be an odd prime, $r$ be an even number and $\F_{p^r}/\F_{p}$ be a finite field extension of degree $r$. 
Let $N_{r}$ be the kernel of the norm map of the extension $\F_{p^r}/\F_{p}$. $N_{r}$ acts $\F_{p^r}$ by multiplication. 
Then, we consider the semidirect product group $N_{r} \ltimes \F_{p^r}$. 
The semidirect product group $N_{r} \ltimes \F_{p^r}$ acts $\F_{p^r}$ naturally. 
Fix $0 \in \F_{p^r}$. The stabilizer of $0$ is $N_{r} \ltimes \{0\} \simeq N_{r}$. Then, we have the $N_{r} \ltimes \F_{p^r}$-orbit decomposition of $\F_{p^r} \times \F_{p^r}$ as follows: 
	\begin{align}\label{orbit decomp wlg}
		\F_{p^r} \times \F_{p^r}
		=\bigsqcup_{i \in \F_{p} }\Delta_{i}, 
	\end{align}
where for each $i \in \F_{p}^{\times}$, $\Delta_{i}=\{ (x, y) \in \F_{p^r} \times \F_{p^r} \mid N_{\F_{p^r}/\F_{p}}(y-x) =i \}$ and $\Delta_{0}=\{(x, x) \in \F_{p^r} \times \F_{p^r} \}$. 
Since $p$ is an odd and $r$ is an even, each $N_{r} \ltimes \F_{p^r}$-orbit is symmetric. 
Therefore, the pair $(N_{r} \ltimes \F_{p^r}, N_{r} \ltimes \{0\})$ is a (symmetric) Gelfand pair. By Corollary \ref{construction thm3}, for each $l \in \{0, \dots, p-1\}$, the graph $X_{\Delta_{l}}$ is a highly-regular graph. 
We denote $X_{\Delta_{l}}$ by $WL(p, r, l)$ in this paper. 
Here, we note that we can easily compute the $N_{r} \ltimes \F_{p^r}$-irreducible decomposition of the $\ell^2$-space $\ell^{2}(\F_{p^r})$ by using (\ref{orbit decomp wlg}). 
Therefore, we get Kloosterman sums as the spherical functions and several formulas corresponding to formulas of spherical functions such as convolution property and addition theorem. Moreover, we get a formula of Kloosterman sums by using the fact that they are simultaneous eigenfunctions of the intersection matrices. 
\begin{remark}
We can apply the above ways to give several formulas of character sums arising as spherical functions including Gauss periods and Kloosterman sums ({\it cf.}~\cite{bannai2009}, \cite{kwok}). Moreover, we note that both eigenvalues and eigenvectors of $B_{l}$ are expressed by the same spherical functions. This is an interesting property of highly-regular graphs constructed by using Corollary \ref{construction thm3}. 
\end{remark}

\section{Basic properties of the elements of a CAM}

In Section $3$, we showed that a connected highly-regular graph has the index $\diam(\Gamma)+1$ if and only if it is a distance-regular graph. Naturally, hence, we can regard the elements of a CAM as generalized constants of the intersection numbers of distance-regular graphs. 

Let $\Gamma$ be a highly-regular graph which satisfies the condition in Theorem \ref{characterization thm1}. 
The following are well-known basic properties of the intersection numbers of a distance-regular graph: 
\begin{itemize}
\item $k=b_{0} \geq b_{1} \geq \cdots \geq b_{m-2} \geq 1$. 
\item $1=c_{1} \leq c_{2} \leq \cdots \leq c_{m-1} \leq k$. 
\end{itemize}

In this section, we give a generalization of the above basic properties of intersection numbers of a distance-regular graph. 

Let $\Gamma$ be a connected highly-regular graph with CAM $C=[c_{i,j}]_{1 \leq i,j \leq m}$ and the valency $k$. Here, let a labeling of $C$ be a labeling with respect to a distance. By Proposition \ref{1, prop3}, for each $i \in \{0, 1, \dots, \diam(\Gamma)\}$, there exists a nonempty subset $S_{i}$ of $I=\{1, \dots, m\}$ such that for any $u \in V$, $D_{i}(u)=\bigsqcup_{t \in S_{i}}V_{t}(u)$. 
For each $i \in \{1, \dots, \diam(\Gamma) \}$, let the integers $b_{i-1}^{\max}$,  $c_{i}^{\max}$, and $c_{i}^{\min}$ be the following: 
\begin{itemize}
\item $b_{i-1}^{\max}=\max \{ \sum_{t \in S_{i}}c_{t,l} \mid l \in S_{i-1} \}$.
\item $c_{i}^{\max}=\max \{ \sum_{t \in S_{i-1}}c_{t,l} \mid l \in S_{i} \}$. 
\item $c_{i}^{\min}=\min \{ \sum_{t \in S_{i-1}}c_{t,l} \mid l \in S_{i} \}$. 
\end{itemize}

Then, we get the following proposition. 

\begin{proposition}\label{intersection numbers1}
	We have the following inequalities: 
	\begin{enumerate}[label={\rm (\arabic*)}]
	\item $k=b_{0}^{\max} \geq b_{1}^{\max} \geq \cdots \geq b_{\diam(\Gamma)-1}^{\max} \geq 1$. 
	\item $1=c_{1}^{\min} \leq c_{2}^{\min} \leq \cdots \leq c_{\diam(\Gamma)}^{\min} \leq k$. 
	\end{enumerate}
\end{proposition}

\begin{proof}
$(1)$
For $i=1$, it is clear that $b_{0}^{\max}$ is equal to $k$. For $i \geq 1$, we take arbitrary $y \in V(\Gamma)$ and $l \in S_{i}$. Then, there exist elements $z \in D_{1}(y)$ and $s \in S_{i-1}$ such that $V_{l}(y)\cap V_{s}(z) \neq \emptyset$. We take arbitrary $x \in V_{l}(y)\cap V_{s}(z)$. First, we show that $D_{1}(x) \cap D_{i+1}(y) \subset D_{1}(x) \cap D_{i}(z)$. 
We take $w \in D_{1}(x)\cap D_{i+1}(y)$. The distance $d(z, w)$ is less than or equal to $d(z, x)+d(x, w)=i$. On the other hand, the distance $d(z, w)$ is greater than or equal to $i$ since  the distance $d(z, w)+d(z, y)$ is greater than or equal to the distance $d(w, y)=i+1$. Hence, the element $w$ is in $D_{1}(x) \cap D_{i}(z)$. Then, we have $D_{1}(x) \cap D_{i+1}(y) \subset D_{1}(x) \cap D_{i}(z)$. By using this, we have 
	\begin{align*}
		\sum_{t \in S_{i+1}}c_{t, l} \leq \sum_{t \in S_{i}}c_{t,s}. 
	\end{align*}
Then, for any $l \in S_{i}$, we have 
	\begin{align*}
		\sum_{t \in S_{i+1}}c_{t, l} \leq b_{i-1}^{\max}. 
	\end{align*}
Therefore, we have $b_{i}^{\max} \leq b_{i-1}^{\max}$.  

$(2)$ For $i=1$, it is clear that $c_{1}^{\min}$ is equal to $1$. For $i \geq 1$, we take arbitrary $z \in V(\Gamma)$ and $s \in S_{i+1}$. Then, there exist $y \in D_{1}(z)$ and $l \in S_{i}$ such that $V_{s}(z) \cap V_{l}(y) \neq \emptyset$. We take arbitrary $x \in V_{l}(y) \cap V_{s}(z)$. First, we show that $D_{1}(x) \cap D_{i-1}(y) \subset D_{1}(x) \cap D_{i}(z)$. 
We take $w \in D_{1}(x) \cap D_{i-1}(y)$. The distance $d(w, z)$ is less than or equal to $i$ since $d(w, z)$ is less than or equal to $d(y, z)+d(y, w)$. On the other hand, the distance $d(w, z)$ is greater than or equal to $i$ since $d(w, z)+d(w, x)$ is greater than or equal to $d(z, x)$. Hence, the element $w$ is in $D_{1}(x) \cap D_{i}(z)$ and we have $D_{1}(x) \cap D_{i-1}(y) \subset D_{1}(x) \cap D_{i}(z)$. By using this, we have 
	\begin{align*}
		\sum_{t \in S_{i-1}}c_{t, l} \leq \sum_{t \in S_{i}}c_{t, s}.
	\end{align*}
Then, for any $k \in S_{i+1}$, we have 
	\begin{align*}
		c_{i}^{\min} \leq \sum_{t \in S_{i}}c_{t, s}.
	\end{align*}
Therefore, we have $c_{i}^{\min} \leq c_{i+1}^{\min}$. 
\end{proof}

\begin{remark}
It is clear that the following property holds: 
	\begin{align*}
		\sum_{j=0}^{m}c_{i,j} =k.
	\end{align*}
\end{remark}

\begin{remark}
We note that the following statement may not always hold in general: 
($\star$) For any $u \in V(\Gamma)$, $i \in \{0, 1, \dots, \diam(\Gamma)-1 \}$, $x \in D_{i}(u)$, the set $D_{1}(x) \cap D_{i+1}(u)$ is nonempty set. 
In fact, there is a counter-example of this statement such as the graph $WL(7, 2, 1)$ which we defined in Section $5$ ({\it cf.}~Figure \ref{WL(7, 2, 1)}). The graph $WL(7, 2, 1)$ has the diameter $3$ and for each vertex $u$, there exists a vertex $x$ which is at distance $2$ from the vertex $u$ such that $D_{1}(x) \cap D_{3}(u)=\emptyset$. 
\begin{figure}[htbp]
	\begin{center}
	\includegraphics[clip,scale=0.50,bb=0 0 360 340]{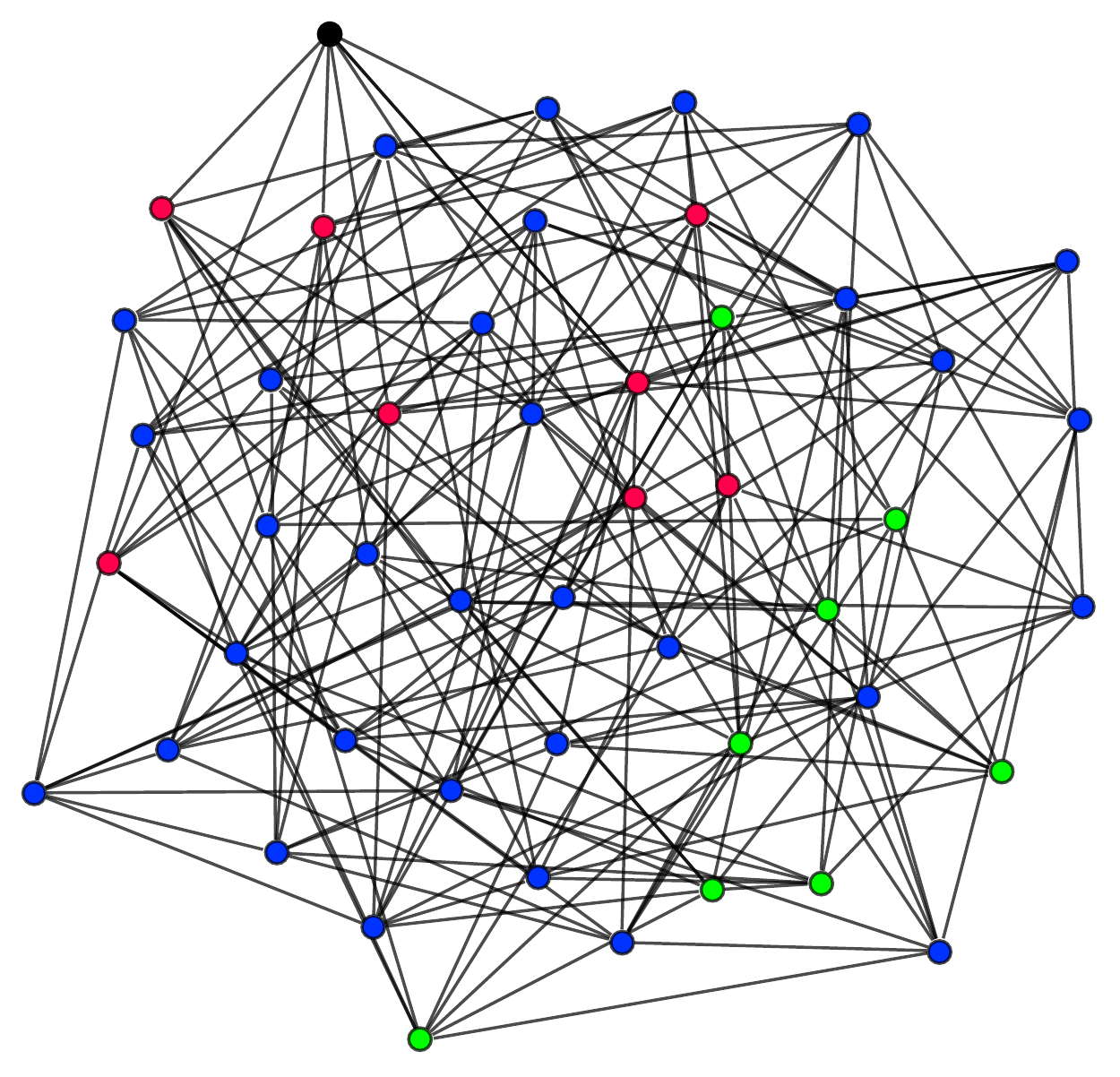}
	\caption{$WL(7, 2, 1)$}
	\label{WL(7, 2, 1)}
	\end{center}
\end{figure}
\end{remark}

The above phenomenon is a difference between highly-regular graphs and distance-regular graphs. 
For a highly-regular graph which satisfies the above statement $(\star)$, the integers $b_{i-1}^{\max}$ and $c_{i}^{\max}$ are satisfied the following property. 

\begin{proposition}\label{intersection numbers2}
We have the following inequality: 
	If $i \geq 1$ and $i+j \leq \diam(\Gamma)$, then $c_{i}^{\max} \leq b_{j}^{\max}$. 
\end{proposition}
\begin{proof}
We take arbitrary $y \in V(\Gamma)$, $l \in S_{i}$ and $x \in V_{l}(y)$. 
By the assumption $(\star)$, there exists the shortest path from $x$ to some element $z \in D_{i+j}(y)$. Then, there exists $s \in S_{j}$ such that $x \in V_{s}(z)$. First, we show that $D_{1}(x)\cap D_{i-1}(y) \subset D_{1}(x) \cap D_{j+1}(z)$. 
We take $w \in D_{1}(x) \cap D_{i-1}(y)$. The distance $d(w, z)$ is less than or equal to $1+j$ since $d(w, x)+d(x, z)$ is greater than or equal to $d(w, z)$. 
On the other hand, the distance $d(w, z)$ is greater than or equal to $1+j$ since $d(w, z)+d(w, y)$ is greater than or equal to $d(z, y)$. Hence, the element $w$ is in $D_{1}(x) \cap D_{j+1}(z)$. Then, we have $D_{1}(x) \cap D_{i-1}(y) \subset D_{1}(x) \cap D_{j+1}(z)$. 
By using this, we have 
	\begin{align*}
		\sum_{t \in S_{i-1}}c_{t, l} \leq \sum_{t \in S_{j+1}}c_{t, s}. 
	\end{align*}
Then, for any $l \in S_{i}$, we have 
	\begin{align*}
		\sum_{t \in S_{i-1}}c_{t, l} \leq b_{j}^{\max}. 
	\end{align*}
Therefore, we have $c_{i}^{\max} \leq b_{j}^{\max}$. 
\end{proof}

By Propositions \ref{intersection numbers1}, \ref{intersection numbers2}, we conclude the following theorem. 
\begin{theorem}\label{intersection numbers}
For a connected highly-regular graph, we have the following inequalities: 
	\begin{enumerate}[label={\rm (\arabic*)}]
	\item $k=b_{0}^{\max} \geq b_{1}^{\max} \geq \cdots \geq b_{m-1}^{\max} \geq 1$. 
	\item $1=c_{1}^{\min} \leq c_{2}^{\min} \leq \cdots \leq c_{m}^{\min} \leq k$. 
	\end{enumerate}
Moreover, we suppose that $\Gamma$ satisfies the following property: 
$(\star)$ For any $u \in V(\Gamma)$, $i \in \{0, 1, \dots, \diam(\Gamma)-1 \}$, $x \in D_{i}(u)$, the set $D_{1}(x) \cap D_{i+1}(u)$ is nonempty set. 
Then, the following inequality holds: 
	\begin{enumerate}[label={\rm (\arabic*)}]
	\item[$(3)$] If $i \geq 1$ and $i+j \leq \diam(\Gamma)$, then $c_{i}^{\max} \leq b_{j}^{\max}$. 
	\end{enumerate}
\end{theorem}

We give an example of Theorem \ref{intersection numbers}. 
Let $C$ be the following matrix: 
	\begin{align*}
		C=\left(
			\begin{array}{cccccc}
				0 & 1 & 0 & 0 & 0 & 0 \\
				4 & 0 & 2 & 1 & 0 & 0 \\
				0 & 2 & 0 & 0 & 1 & 0 \\
				0 & 1 & 0 & 1 & 1 & 0 \\
				0 & 0 & 2 & 2 & 1 & 2 \\
				0 & 0 & 0 & 0 & 1 & 2 
			\end{array}
			\right). 
	\end{align*}
We consider the graph $C_{5} \Box C_{5}$ ({\it cf.}~Figure \ref{ttg55}). Here, $C_{5}$ is the cycle graph of order $5$. Let the labeling of vertices be the same as in Figure \ref{ttg55}. 
It is easy to check that this graph is a highly-regular graph with CAM $C$ which satisfies the condition ($\star$). 
For the vertex $25 \in V(C_{5} \Box C_{5})$, we can take the partition of $V(C_{5} \Box C_{5})$ with respect to $25 \in V(C_{5} \Box C_{5})$ as follows: 
	\begin{align*}
		&V_{0}(25)=\{25\}, \\
		&V_{1}(25)=\{5, 20, 21, 24\}, \\
		&V_{2_{1}}(25)=\{1, 4, 16, 19\}, V_{2_{2}}(25)=\{10, 15, 22, 23\}, \\
		&V_{3}(25)=\{2, 3, 6, 9, 11, 14, 17, 18\}, \\
		&V_{4}(25)=\{7, 8, 12, 13\}. 
	\end{align*}
The entries of the matrix $C$ satisfy the inequalities in Theorem \ref{intersection numbers}. 
\begin{figure}[htbp]
	\begin{center}
	\includegraphics[clip,scale=0.50,bb=0 0 370 360]{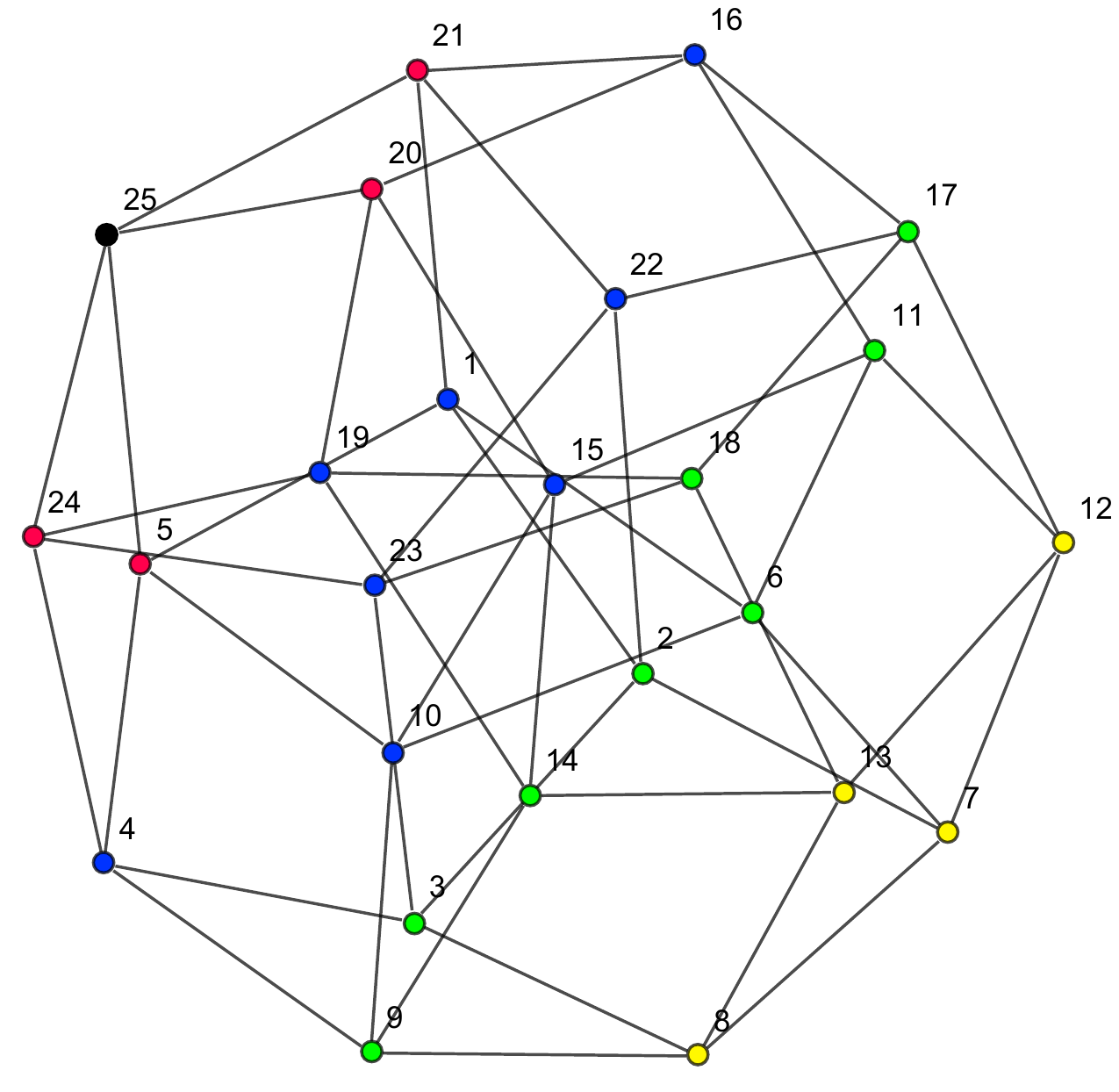}
	\caption{$C_{5} \Box C_{5}$}
	\label{ttg55}
	\end{center}
\end{figure}

\appendix
\section{Infinite families of connected highly-regular graphs with fixed valency which are not distance-regular graphs}
In this section, we give infinite families of connected highly-regular graphs with fixed valency which are not distance-regular graphs explicitly. 
More precisely, we discuss the following question: 
\begin{question}\label{fixed valency question}
Are there only finitely many connected highly-regular graphs of fixed valency greater than 2 which are not distance-regular graphs?
\end{question} 

First, we construct highly-regular graphs of the valency $3$ and $4$ which are not distance-regular graphs explicitly. 

Let $n$, $m$ be positive integers greater than 1. We denote the graph $C_{n} \Box C_{m}$ by $T_{n,m}$, where $C_{n}$ and $C_{m}$ are cycle graphs of order $n$ and $m$ respectively. Here, we note that $C_{2}$ is in Figure \ref{c2}. Also, we note that the graph $T_{n,m}$ is vertex-transitive. 
Without loss of generality, We may assume $n \leq m$. 
\begin{figure}[htbp]
	\begin{center}
	\includegraphics[clip,scale=0.4,bb=0 0 360 30]{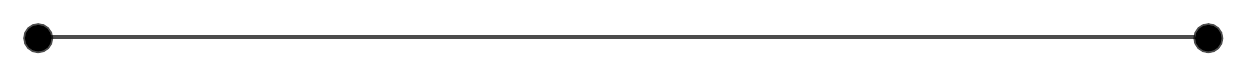}
	\caption{$C_{2}$}
	\label{c2}
	\end{center}
\end{figure}

By Proposition \ref{1, prop6}, the graph $T_{n,m}$ is a connected highly-regular graph since a cycle graph is a distance-regular graph. 
Then, we have the following theorem. 
\begin{proposition}\label{infinite families valency 3,4}
The graph $T_{n,m}$ is a connected highly-regular graph which is not a distance-regular graph except for the cases $(n, m)=(2, 2), (2, 4), (3, 3), (4, 4)$. 
\end{proposition}

\begin{proof}
In the case $n=2$, $T_{2,m}$ has the valency $3$. If $m=2$, $T_{2,2}$ is the cycle graph of order $4$. Hence, $T_{2,2}$ is a distance-regular graph. If $m=3$, we take $\pmb{v} \in V(T_{2,3})$. Then, the cardinality of the degree set of $\langle D_{1}(\pmb{v}) \rangle$ is $2$. By Proposition \ref{1, prop5}, $\I(T_{2,3})$ is greater than $1+\diam(T_{2,3})$. By Corollary \ref{characterization cor1}, $T_{2,3}$ is not a distance-regular graph. If $m=4$, $T_{2,4}$ is the Hamming graph of order $8$ (cube). Hence, $T_{2,4}$ is a distance-regular graph. 
In the case $m>4$, we take arbitrary $\pmb{v} \in V(T_{2,m})$. If $D_{1}(\pmb{v})$ is divided , the graph $T_{2,m}$ is not a distance-regular graph. If $D_{1}(\pmb{v})$ is not divided, at least $D_{2}(\pmb{v})$ is divided into the following as a partition in a highly-regular graph: 
	\begin{align*}
		D_{2}(\pmb{v})=\{ \pmb{w} \in D_{2}(\pmb{v}) \mid |D_{1}(\pmb{w}) \cap D_{1}(\pmb{v})| =2 \} \sqcup \{ \pmb{w} \in D_{2}(\pmb{v}) \mid |D_{1}(\pmb{w}) \cap D_{1}(\pmb{v})|=1 \}. 
	\end{align*}
By Corollary \ref{characterization cor1}, the graph $T_{2,m}$ is not a distance-regular graph. 

In the case $n=3$, $T_{3,m}$ has the valency $4$. If $m=3$, $T_{3,3}$ is the graph as in Figure \ref{ttg33}, and $T_{3,3}$ is a distance-regular graph. 
If $m$ is greater than $3$, we take arbitrary $\pmb{v} \in V(T_{3,m})$. Then, at least $D_{1}(\pmb{v})$ is divided into the following as a partition of a vertex set in a highly-regular graph: 
	\begin{align*}
		D_{1}(\pmb{v})= \{\pmb{w} \in D_{1}(\pmb{v}) \mid |D_{1}(\pmb{w}) \cap D_{1}(\pmb{v})|=1 \} \sqcup 
					\{ \pmb{w} \in D_{1}(\pmb{v}) \mid |D_{1}(\pmb{w}) \cap D_{1}(\pmb{v})|=0 \}. 
	\end{align*}
By Corollary \ref{characterization cor1}, the graph $T_{3,m}$ is not a distance-regular graph. 

In the case $n=4$, $T_{4,m}$ has the valency $4$. If $m=4$, $T_{4,4}$ is the Hamming graph of order $16$. Hence, $T_{4,4}$ is a distance-regular graph. 
If $m$ is greater than $4$, we take arbitrary $\pmb{v} \in V(T_{4,m})$. 
If $D_{1}(\pmb{v})$ is divided, the graph $T_{4, m}$ is not a distance-regular graph. 
If $D_{1}(\pmb{v})$ is not divided, at least $D_{2}(\pmb{v})$ is divided into the following as a partition in a highly-regular graph: 
	\begin{align*}
		D_{2}(\pmb{v})=\{ \pmb{w} \in D_{2}(\pmb{v}) \mid |D_{1}(\pmb{v}) \cap D_{1}(\pmb{w})|=1 \} \sqcup \{ \pmb{w} \in D_{2}(\pmb{v}) \mid |D_{1}(\pmb{v}) \cap D_{1}(\pmb{w})|=2 \}. 
	\end{align*}
By Corollary \ref{characterization cor1}, the graph $T_{4,m}$ is not a distance-regular graph. 

In the case $n>4$, we take arbitrary $\pmb{v} \in V(T_{n,m})$. If $D_{1}(\pmb{v})$ is divided, the graph $T_{n,m}$ is not a distance-regular graph. 
If $D_{1}(\pmb{v})$ is not divided, at least $D_{2}(\pmb{v})$ is divided into the following as a partition in a highly-regular graph: 
	\begin{align*}
		D_{2}(\pmb{v})=\{ \pmb{w} \in D_{2}(\pmb{v}) \mid |D_{1}(\pmb{v}) \cap D_{1}(\pmb{w})|=1 \} \sqcup \{ \pmb{w} \in D_{2}(\pmb{v}) \mid |D_{1}(\pmb{v}) \cap D_{1}(\pmb{w})|=2 \}.
	\end{align*}
By Corollary \ref{characterization cor1}, the graph $T_{n,m}$ is not a distance-regular graph. 
\end{proof}

\begin{remark}
As we mentioned above, $T_{2, m}$ $(m > 2)$ has the valency $3$, and $T_{n,m}$ $(n \geq 3)$ has the valency $4$. 
\end{remark}

\begin{remark}
The graph $T_{n,m}$ has the diameter $[\frac{n}{2}]+[\frac{m}{2}]$. 
Here, the symbol $[x]$ denotes the largest integer less than $x$.  
Therefore, $T_{n,m}$ has the diameter $2$ if and only if $(n,m)=(2,2), (2,3), (3,3)$. 
Then, we can easily observe the following:
\begin{itemize}
\item $\overline{T_{2,2}}$ is not connected. 
\item $\overline{T_{2,3}}$ is the cycle graph of order $6$. 
\item $\overline{T_{3,3}} \simeq T_{3,3}$ is a distance-regular graph ({\it cf.}~Figure \ref{ttg33}). 
\end{itemize}
\begin{figure}[htbp]
	\begin{center}
	\includegraphics[clip,scale=0.4,bb=0 0 360 340]{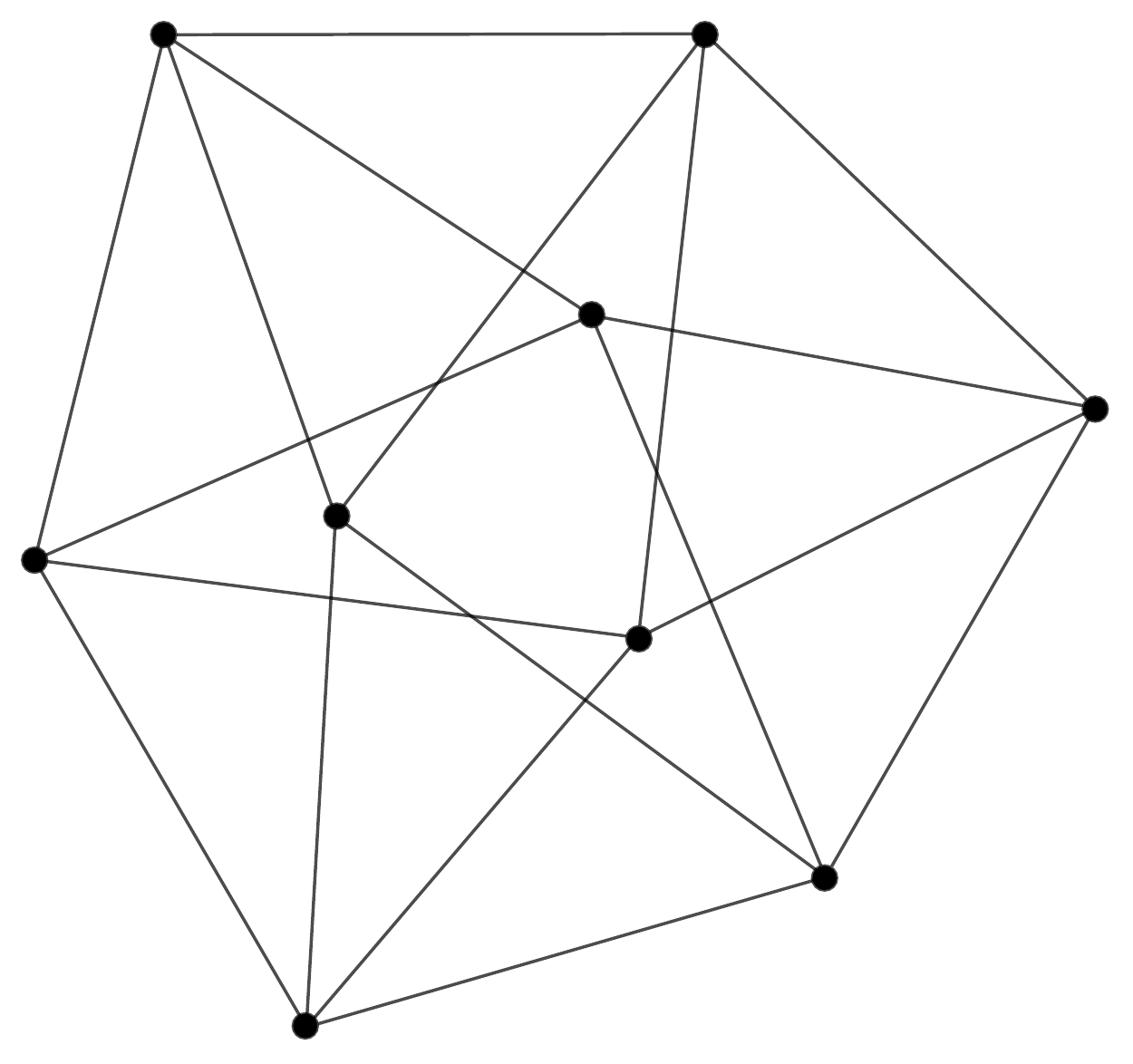}
	\caption{$T_{3, 3} \simeq \overline{T_{3,3}}$}
	\label{ttg33}
	\end{center}
\end{figure}
\end{remark}

Next, we construct some infinite families of connected highly-regular graphs of fixed valency greater than $4$ which are not distance-regular graphs. 

\begin{proposition}\label{Cartesian product of highly-regular graphs}
Let $\Gamma_{1}$ be a connected highly-regular graph which is not a distance-regular graph and $\Gamma_{2}$ be a connected highly-regular graph. Then, the Cartesian product $\Gamma_{1} \Box \Gamma_{2}$ is a connected highly-regular graph which is not a distance-regular graph. 
\end{proposition}

\begin{proof}
First, we note that $\diam(\Gamma_{1} \Box \Gamma_{2})=\diam(\Gamma_{1})+\diam(\Gamma_{2})$. 
By Proposition \ref{1, prop6}, the Cartesian product $\Gamma_{1} \Box \Gamma_{2}$ is a highly-regular graph. 
We take arbitrary $\pmb{v}=(v_{1}, v_{2}) \in V(\Gamma_{1}) \times V(\Gamma_{2})$. 
For each $j \in \{0, 1, \dots, \diam(\Gamma_{1})+\diam(\Gamma_{2}) \}$, 
	\begin{align*}
		&D_{j, \Gamma_{1} \Box \Gamma_{2}}(\pmb{v}) \\
			&=\bigsqcup_{k, l} \{ (v, w) \in V(\Gamma_{1}) \times V(\Gamma_{2}) \mid v \in D_{k, \Gamma_{1}} (v_{1})~and~w \in D_{l, \Gamma_{2}}(v_{2}) \}, 
	\end{align*}
where $k$ and $l$ run through $0 \leq k \leq \diam(\Gamma_{1})$, $0 \leq l \leq \diam(\Gamma_{2})$ such that $k+l=j$. 
Since $\Gamma_{1}$ is not a distance-regular graph, there exist $i \in \{ 0, \dots, \diam(\Gamma_{1}) \}$, $j \in \{ i-1, i, i+1 \}$, $u_{1}, u_{2} \in D_{i, \Gamma_{1}}(v_{1})$ such that
	\begin{eqnarray}\label{cp1}
		|D_{1, \Gamma_{1}}(u_{1}) \cap D_{j, \Gamma_{1}}(v_{1})| \neq 
		|D_{1, \Gamma_{1}}(u_{2}) \cap D_{j, \Gamma_{1}}(v_{1})|. 
	\end{eqnarray}
We consider the vertices $(u_{1}, v_{2}), (u_{2}, v_{2}) \in D_{i, \Gamma_{1} \Box \Gamma_{2}}(\pmb{v})$. 
Then, $D_{j, \Gamma_{1}}(v_{1}) \cap D_{1, \Gamma_{1}}(u_{1})$ and $D_{j, \Gamma_{1}}(v_{1}) \cap D_{1, \Gamma_{1}}(u_{2})$ are divided into the following: 
	\begin{align*}
		&D_{1, \Gamma_{1} \Box \Gamma_{2}}((u_{1},v_{2})) \cap D_{j, \Gamma_{1} \Box \Gamma_{2}}(\pmb{v})\\
		&=\{ (u_{1}, w) \in V(\Gamma_{1}) \times V(\Gamma_{2}) \mid \{w, v_{2}\} \in E(\Gamma_{2})~{\rm and}~d(w, v_{2})=j-i \} \\
			& \hspace{0.5cm} \sqcup 
			\{ (v, v_{2}) \in V(\Gamma_{1}) \times V(\Gamma_{2}) \mid v \in D_{1, \Gamma_{1}}(u_{1}) \cap D_{j, \Gamma_{1}}(v_{1}) \}, \\
		&D_{1, \Gamma_{1} \Box \Gamma_{2}}((u_{2},v_{2})) \cap D_{j, \Gamma_{1} \Box \Gamma_{2}}(\pmb{v}) \\
		&=\{ (u_{2}, w) \in V(\Gamma_{1}) \times V(\Gamma_{2}) \mid \{w, v_{2}\} \in E(\Gamma_{2})~{\rm and}~d(w, v_{2})=j-i \} \\
			&\hspace{0.5cm} \sqcup 
			\{ (v, v_{2}) \in V(\Gamma_{1}) \times V(\Gamma_{2}) \mid v \in D_{1, \Gamma_{1}}(u_{2}) \cap D_{j, \Gamma_{1}}(v_{1}) \}. 
	\end{align*}
By (\ref{cp1}), we have
	\begin{align*}
		| D_{1, \Gamma_{1} \Box \Gamma_{2}}((u_{1},v_{2})) \cap D_{j, \Gamma_{1} \Box \Gamma_{2}}(\pmb{v}) | \neq
		| D_{1, \Gamma_{1} \Box \Gamma_{2}}((u_{2},v_{2})) \cap D_{j, \Gamma_{1} \Box \Gamma_{2}}(\pmb{v}) |. 
	\end{align*}
Therefore, $\Gamma_{1} \Box \Gamma_{2}$ is not a distance-regular graph. 
\end{proof}

Let $\mathcal{P}_{1}$ be the infinite family of connected highly-regular graphs of fixed valency $3$
and $\mathcal{P}_{2}$ be the infinite family of connected highly-regular graphs of fixed valency $4$ which we construct explicitly in Proposition \ref{infinite families valency 3,4}. 

Let $k$ be an integer greater than $4$. Then, there exist $r_{1}, r_{2}, r_{3} \in \Z_{\geq 0}$ with $(r_{2}, r_{3}) \neq (0, 0)$ such that $k=r_{1}\cdot1+r_{2}\cdot3+r_{3}\cdot4$. 
Let $\mathcal{P}(r_{2})$ and $\mathcal{P}(r_{3})$ be the following: 
\begin{itemize}
\item	$\mathcal{P}(r_{2})=\{ \Box_{j=1}^{r_{2}}\Gamma_{1,j} \mid \Gamma_{1,j} \in \mathcal{P}_{1}, 1 \leq j \leq r_{2} \}.$
\item	$\mathcal{P}(r_{3})=\{ \Box_{l=1}^{r_{3}}\Gamma_{2,l} \mid \Gamma_{2,l} \in \mathcal{P}_{2}, 1 \leq l \leq r_{3} \}.$ 
\end{itemize}
Moreover, let $\mathcal{P}(r_{1}, r_{2}, r_{3})$ be the following: 
\begin{align*}
\mathcal{P}(r_{1}, r_{2}, r_{3})=\{ (C_{2}^{\Box r_{2}}) \Box \Gamma_{1} \Box \Gamma_{2} \mid \Gamma_{1} \in \mathcal{P}(r_{2}), \Gamma_{2} \in \mathcal{P}(r_{3}) \}. 
\end{align*}
By Proposition \ref{Cartesian product of highly-regular graphs}, we have the following proposition. 
\begin{proposition}\label{infinite families valency greater than 4}
The families $\mathcal{P}(r_{1}, r_{2}, r_{3})$ with $r_{1}\cdot1+r_{2}\cdot3+r_{3}\cdot4=k$, $(r_{2}, r_{3}) \neq (0, 0)$ are infinite families of connected highly-regular graphs of fixed valency $k$ which are not distance-regular graphs. 
\end{proposition}
By Proposition \ref{infinite families valency 3,4}, \ref{infinite families valency greater than 4}, we conclude the following theorem. 
\begin{theorem}
There are infinitely many connected highly-regular graphs of fixed valency greater than $2$  which are not distance-regular graphs.
\end{theorem}

\section{Local spectral properties of highly-regular graphs}
In this section, we give the following local spectral properties of highly-regular graphs. 
For vertices $u, v \in V$ and an eigenvalue $\lambda_{l}$, {\it $uv$-crossed multiplicity $m_{uv}(\lambda_{l})$} is $uv$-entry of $E_{l}$. 
Let $\Gamma$ be a connected highly-regular graph with CAM $C=[c_{i,j}]_{1\leq i, j \leq m}$. 
Here, let a labeling of $C$ with respect to a distance. 
For each $i \in \{0, 1, \dots, \diam(\Gamma)\}$, there exists a nonempty subset $S_{i}$ of $I=\{1, \dots, m\}$ such that $D_{i}(u)=\bigsqcup_{t \in S_{i}}V_{t}(u)$, $u \in V$. 
Then, we have the following theorem. 
\begin{theorem}
A connected highly-regular graph $\Gamma$ is a spectrally-regular graph. 
Moreover, for each $u \in V$ and for two vertices $v, w \in V_{s}(u_{0})$, $s \in S_{j}$,  we have $m_{uv}(\lambda_{l})=m_{uw}(\lambda_{l})$, for any $\lambda_{l}$. 
\end{theorem}
\begin{proof}
We fix a vertex $u \in V$. We define the matrix $P^{u} \in M_{mn}(\R)$ whose entries are given by 
\begin{align*}
(P^{u})_{t,w}=\begin{cases}
				1 & \text{if $w \in V_{t}(u)$, } \\
				0 & \text{otherwise. } 
			\end{cases}
\end{align*}
It is easy to check that this matrix intertwines the adjacency matrix $A$ and collapsed adjacency matrix $C$, that is, $P^{u}A=CP^{u}$. 
For each eigenvalue $\lambda_{l}$, there exists unique polynomial such that $E_{l}=Z_{l}(A)$ by using Lagrange interpolation. 
This implies that for each $l=0, \dots, d$, we have $P^{u}Z_{l}(A)=Z_{l}(C)P^{u}$. 
For $u \in V$, $t \in S_{i}$, $s \in S_{j}$, $w \in V_{s}(u)$, we have 
\begin{align}\label{ac}
\sum_{z \in V_{t}(u)}Z_{l}(A)_{z, w}=Z_{l}(C)_{t,s}. 
\end{align}
Putting $s=t=1 \in S_{0}$, $w \in V_{1}(u)$, we have 
\begin{align*}
m_{u}(\lambda_{l})=Z_{l}(A)_{u,u}=Z_{l}(C)_{1,1}. 
\end{align*}
Therefore, $\Gamma$ is a spectrally-regular graph. 
Moreover, putting $t=1$ in $(\ref{ac})$, we have 
\begin{align*}
Z_{l}(A)_{u,w}=Z_{l}(C)_{1,s}. 
\end{align*}
This implies that for $v, w \in V_{s}(u)$, $m_{u,v}(\lambda_{l})=Z_{l}(C)_{1,s}=m_{u,w}(\lambda_{l})$. 
\end{proof}

\section*{Acknowledgment}
The author expresses gratitude to Professor Hiroyuki Ochiai for his many helpful comments.


\begin{thebibliography}{20}
\bibitem{alavi} Y.~Alavi, G.~Chartrand, D.~R.~Lick and H.~C.~Swart, Highly regular graphs, Ann.~New York Acad.~Sci.~$576$ $(1989)$ $20$--$29$. 

\bibitem{angel} J.~Angel, Finite upper half planes over finite fields, Finite Fields Appl.~$2$ $(1996)$ $62$--$86$. 

\bibitem{bang} S.~Bang, A.~Dubickas, J.~H.~Koolen and V.~Moulton, There are only finitely many distance-regular graphs of fixed valency greater than two, Adv.~Math.~$269$ $(2015)$ $1$--$55$. 

\bibitem{bannai} E.~Bannai and T.~Ito, Algebraic Combinatorics I: Association Schemes, Benjamin/Cummings Publishing Co.~Inc., Menro Park, CA, $1984$. 

\bibitem{bannai2009} E.~Bannai, O.~Shimabukuro and H.~Tanaka, Finite Euclidean graphs and Ramanujan graphs, Discrete Math.~$309$ $(2009)$ $6126$--$6134$. 

\bibitem{bollobas} B.~Bollob{\'a}s, Graph theory: An Introductory Course, Springer-Verlag, Berlin, $1979$. 

\bibitem{bollobas1}B.~Bollob{\'a}s, Modern graph theory, Springer-Verlag, Berlin, $2002$. 

\bibitem{brouwer} A.~E.~Brouwer, A.~M.~Cohen and A.~Neumaier, Distance-regular graphs, volume 18 of Ergebnisse der Mathematik und ihrer Grenzgebiete $(3)$ $[$Results in Mathematics and Related Areas $(3)$$]$. Springer-Verlag, Berlin, $1989$. 

\bibitem{cameron} P.~J.~Cameron, There are only finitely many finite distance-transitive graphs of given valency greater than two, Combinatorica $2$ $(1)$ $(1982)$ $9$--$13$. 

\bibitem{cameronpraegersaxlseitz} P.~J.~Cameron, C.~E.~Praeger, J.~Saxl and G.~M.~Seitz, On the Sims conjecture and distance transitive graphs, Bull.~Lond.~Math.~Soc.~$15$ $(5)$ $(1983)$ $499$--$506$. 

\bibitem{cecc} T.~Ceccherini-Silberstein, F.~Scarabotti and F.~Tolli, Representation Theory and Harmonic Analysis of wreath products of finite groups, Cambridge University Press $(2014)$. 

\bibitem{fiol} M.~A.~Fiol, E.~Garriga and J.~L.~A.~Yebra, Locally pseudo-distance-regular graphs, J.~Combin.~Theory Ser.~B $68$ $(1996)$ $179$--$205$. 

\bibitem{kwok} W.~M.~Kwok, Character tables of association schemes of affine type, European J.~Combin.~$13$ $(1992)$ $167$--$185$. 

\bibitem{winnie} W.~Li, Character sums and abelian Ramanujan graphs, J.~Number Theory $41$ $(1992)$ $199$--$217$. 

\bibitem{terras} A.~Terras, Fourier analysis of finite groups and applications, London Mathematical Society Student Texts, $43$, Cambridge University Press, Cambridge, $1999$. 
\end{thebibliography}
\end{document}